%% file: SW4.tex
\documentclass{amsart}

\usepackage{amsmath,amsthm,amssymb,amscd,here}

\newtheorem{thm}{Theorem}
\newtheorem{prop}[thm]{Proposition}
\newtheorem{lem}[thm]{Lemma}
\newtheorem{cor}[thm]{Corollary}

\theoremstyle{definition}

\theoremstyle{remark}
\theoremstyle{definition}
\newtheorem{prob}[thm]{Problem}

\theoremstyle{remark}
\newtheorem{rem}[thm]{Remark}

\newcommand{\Hom}{{\rm Hom}}

\title{Answers to some problems about graph coloring test graphs}

\begin{document}

\author{Takahiro Matsushita}


\begin{abstract}
We prove that a graph whose chromatic number is 2 is a homotopy test graph. We also prove that there is a graph $K$ with two involutions $\gamma_1$ and $\gamma_2$ such that $(K,\gamma_1)$ is a Stiefel-Whitney test graph, but $(K,\gamma_2)$ is not. These are answers to some of the problems suggested by Kozlov.
\end{abstract}

\maketitle

\section{Introduction}

One of the most remarkable applications of algebraic topology to combinatorics is Lov$\acute{\rm a}$sz' proof of the Kneser conjecture. In \cite{Lov1}, Lov$\acute{\rm a}$sz introduced the neighborhood complex $N(G)$ of a graph $G$, and showed that the connectivity of $N(G)$ gives a lower bound for the chromatic number of $G$.

The Hom complexes of graphs were defined by Lov$\acute{\rm a}$sz and were first mainly investigated by Babson and Kozlov in \cite{BK1} and \cite{BK2}. The Hom complex ${\rm Hom}(K_2,G)$ from the complete graph $K_2$ with 2-vertices to a graph $G$ is homotopy equivalent to the neighborhood complex $N(G)$ of $G$. Kozlov introduced the notion of homotopy and Stiefel-Whitney test graphs in \cite{Koz1}. A test graph is a graph $T$ such that a certain inequality holds between a topological invariant of ${\rm Hom}(T,G)$ and the chromatic number of $G$. For example, the graph $K_2$ is a homotopy test graph. The conjecture by Lov$\acute{\rm a}$sz states that odd cycles are homotopy test graphs, and this conjecture was solved by Babson and Kozlov in \cite{BK2}.

In \cite{Koz1} Kozlov suggested a number of problems about test graphs. The purpose of this paper is to solve two of them:

\begin{prob}[Conjecture 6.2.1 of \cite{Koz1}\footnote{The precise statement of the conjecture of \cite{Koz1} is ``Every connected bipartite graph is a homotopy test graph". Here we should consider bipartite graphs as graphs with $\chi(G) =2$ since graphs with $\chi(G) \leq 1$ are clearly not homotopy test graphs.}]
Is a graph $G$ with $\chi(G) = 2$ a homotopy test graph?
\end{prob}

\begin{prob}[Section 6.1 of \cite{Koz1}]
Does there exist a graph $K$ having two different flipping involutions $\gamma_1$ and $\gamma_2$, such that $(K,\gamma_1)$ is a Stiefel-Whitney test graph, whereas $(K,\gamma_2)$ is not?
\end{prob}

The following theorems answer the above problems.

\begin{thm}
A graph whose chromatic number is $2$ is a homotopy test graph.
\end{thm}

\begin{thm}
There is a graph $K$ with two involutions $\gamma_1$ and $\gamma_2$ such that $(K,\gamma_1)$ is a Stiefel-Whitney test graph, but $(K,\gamma_2)$ is not.
\end{thm}

The following figure describes the graph $K$ and the involutions $\gamma_1$ and $\gamma_2$.

\vspace{2mm}
\begin{center}
\begin{picture}(70,70)(0,5)
\put(0,40){\circle*{2.5}} \put(10,60){\circle*{2.5}} \put(10,20){\circle*{2.5}} \put(20,30){\circle*{2.5}} \put(20,50){\circle*{2.5}}
\put(40,30){\circle*{2.5}} \put(40,50){\circle*{2.5}} \put(50,20){\circle*{2.5}} \put(50,60){\circle*{2.5}} \put(60,40){\circle*{2.5}}

\put(0,40){\line(1,2){10}} \put(0,40){\line(1,-2){10}} \put(10,60){\line(1,-1){10}} \put(10,20){\line(1,1){10}} \put(20,50){\line(1,0){20}}
\put(20,50){\line(0,-1){20}} \put(20,30){\line(1,0){20}} \put(40,30){\line(0,1){20}} \put(40,30){\line(1,-1){10}} \put(40,50){\line(1,1){10}}
\put(50,60){\line(1,-2){10}} \put(50,20){\line(1,2){10}}

\put(70,40){\vector(0,1){20}} \put(70,40){\vector(0,-1){20}}
\put(30,68){\vector(1,0){20}} \put(30,68){\vector(-1,0){20}}

\put(27,76){$\gamma_2$} \put(75,38){$\gamma_1$} \put(-17,38){$K$} \put(10,4){\bf Figure 1.}
\end{picture}
\end{center}
The involution $\gamma_1$ is the reflection in the horizontal line, and the involution $\gamma_2$ is the reflection in the vertical line.

\vspace{3mm}
\noindent {\bf Acknowledgement.} The author thanks to Toshitake Kohno for helpful suggestions. He would like to express his gratitude to the referees. Their valuable comments and suggestions made the manuscript much improved and readable. The author was supported by the Grant-in-Aid for Scientific Research (KAKENHI No. 25-4699) and the Grant-in-Aid for JSPS fellows. This work was supported by the Program for Leading Graduate Schools, MEXT, Japan.

\section{Preliminaries}
In this section, we review the definitions and known results about test graphs. We refer to \cite{Koz1} and \cite{Koz2} for the more concrete introduction to the subject.

A graph is a pair $G = (V(G),E(G))$ where $V(G)$ is a finite set and $E(G)$ is a symmetric subset of $V(G) \times V(G)$. Hence the graphs considered are finite, non-directed, have no parallel edges, but may have loops. For graphs $G$ and $H$, a graph homomorphism from $G$ to $H$ is a map $f: V(G) \rightarrow V(H)$ such that $(f \times f) (E(G)) \subset E(H)$. For a non-negative integer $n$, we write $K_n$ to indicate the complete graph with $n$-vertices. Then the chromatic number of a graph $G$ is formulated as the number
$$\chi (G) = \inf \{ n \geq 0\; | \textrm{ There is a graph homomorphism $G \rightarrow K_n$.}\}.$$
In this paper, we assume that the infimum (or the supremum) of the empty set is $+\infty$ (or $-\infty$, respectively). Thus if $G$ has a loop, then we consider $\chi(G) = + \infty$.

Let $G$ and $H$ be graphs. A multi-homomorphism from $G$ to $H$ is a map $\eta : V(G) \rightarrow 2^{V(H)} \setminus \{ \emptyset\}$ such that for $(v,w) \in E(G)$, we have $\eta(v) \times \eta(w) \subset E(H)$. For two multi-homomorphisms $\eta$ and $\eta'$ from $G$ to $H$, we write $\eta \leq \eta'$ if $\eta(v) \subset \eta' (v)$ for any $v \in V(G)$. The poset of all multi-homomorphisms from $G$ to $H$ with this ordering is called the Hom complex from $G$ to $H$, and is denoted by ${\rm Hom}(G,H)$. For a graph homomorphism $f:T \rightarrow S$, define the order-preserving map $f^* :{\rm Hom}(S,G) \rightarrow {\rm Hom}(T,G)$ by $\eta \mapsto \eta \circ f$.

Note that a graph homomorphism $f:G \rightarrow H$ can be regarded as a multi-homomorphism $V(G) \rightarrow 2^{V(H)} \setminus \{ \emptyset \}, v \mapsto \{ f(v)\}$.

An involution of a graph $T$ is a graph homomorphism $\gamma : T \rightarrow T$ such that $\gamma^2 = {\rm id}_T$. An involution $\gamma$ of $T$ is flipping if there is a vertex $v \in V(T)$ such that $(v,\gamma(v)) \in E(T)$. A pair $(G,\gamma)$ of a graph $G$ and an involution $\gamma$ of $G$ is called a $\mathbb{Z}_2$-graph, and if $\gamma$ is flipping, then the $\mathbb{Z}_2$-graph $(G,\gamma)$ is said to be flipping.

Let $X$ be a free $\mathbb{Z}_2$-CW-complex. We write $\overline{X}$ to indicate the orbit space of $X$. Let $h(X)$ denote the number
$$\sup \{ n \geq 0 \; | \; w_1(X)^n \neq 0\}$$
where $w_1(X) \in H^1(\overline{X} ; \mathbb{Z}_2)$ denotes the 1st Stiefel-Whitney class of the double cover $X \rightarrow \overline{X}$.

Let $(T,\gamma)$ be a flipping $\mathbb{Z}_2$-graph. Then for any loopless graph $G$, the Hom complex $\Hom (T,G)$ becomes a free $\mathbb{Z}_2$-complex with the involution $\eta \leftrightarrow \eta \circ \gamma$, $(\eta \in \Hom (T,G))$. A flipping $\mathbb{Z}_2$-graph $(T,\gamma)$ is called a Stiefel-Whitney test graph\footnote{In \cite{DS1} another definition of a Stiefel-Whitney test graph is employed. In \cite{DS1} a graph $T$ is a ``Stiefel-Whitney test graph" if for any graph $G$, the inequality $\chi(G) \geq h(\Hom (T,G)) + \chi(T)$ holds. It is easy to see that our definition implies theirs, but it was not shown that these definitions are equivalent.} if the following equality holds for $n \geq \chi(T)$:
$$h (\Hom (T,K_n)) = n - \chi (T).$$
 
Let $n$ be an integer with $n \geq -1$. A topological space $X$ is said to be $n$-connected if for any integer $k$ with $-1 \leq k \leq n$ and for any continuous map $f:S^k \rightarrow X$, there is a continuous map $g:D^{k+1} \rightarrow X$ such that $g|_{S^k} = f$. Here we regard $S^{-1}$ as the empty space and $D^0$ as the one point space. The connectivity of the space $X$ is the number
$${\rm conn}(X) = \sup \{ n \geq -1 \; | \; \textrm{$X$ is $n$-connected.}\}.$$
A graph $T$ is called a homotopy test graph if for any graph $G$, the following inequality holds:
$$\chi (G) > {\rm conn}(\Hom (T,G)) + \chi (T).$$

In this paper, we abbreviate ``a homotopy test graph" to ``an HT-graph", and ``a Stiefel-Whitney test graph" to ``an SWT-graph."

In \cite{Koz1} Kozlov showed that if a flipping $\mathbb{Z}_2$-graph $(T,\gamma)$ is an SWT-graph, then $T$ is an HT-graph. It is clear that a graph having no edges is not an HT-graph. The first non-trivial graph which is not an HT-graph was discovered by Hoory and Linial in \cite{HL}. Babson and Kozlov showed in \cite{BK1} that $K_n$ for $n \geq 2$ is an SWT-graph by the involution exchainging $1$ and $2$ and fixing the other vertices. They also proved in \cite{BK2} that odd cycle $C_{2r+1}$ for $r \geq 1$ is an HT-graph. They conjectured that $C_{2r+1}$ with the reflection is an SWT-graph. This conjecture was firstly solved by Schultz in \cite{Sch1}, and other proofs are found in \cite{Koz3} and \cite{Sch2}.

\section{Proof of Theorem 3}

Recall that a space $Y$ is called a retract of a space $X$ if there are continuous maps $i:Y \rightarrow X$ and $r:X \rightarrow Y$ such that $ri = {\rm id}_Y$.

\begin{lem}
Suppose that a space $Y$ is a retract of a space $X$. If $X$ is $n$-connected, then $Y$ is also $n$-connected.
\end{lem}
\begin{proof}
Let $i:Y \rightarrow X$ and $r:X \rightarrow Y$ be continuous maps such that $ri = {\rm id}_Y$. Let $k$ be an integer such that $-1 \leq k \leq n$ and $f:S^k \rightarrow Y$ a continuous map. Since $X$ is $n$-connected, the map $if: S^k \rightarrow X$ can be extended to a continuous map $g:D^{k+1} \rightarrow X$. Then the map $rg: D^{k+1} \rightarrow Y$ is an extension of $f$. Therefore $Y$ is $n$-connected.
\end{proof}

Similarly, a graph $S$ is called a retract of a graph $T$ if there are graph homomorphisms $i:S \rightarrow T$ and $r :T \rightarrow S$ such that $ri = {\rm id}_S$. In this case, we have that $\chi(S) = \chi(T)$ since there are graph homomorphisms from each of them to the other. Let $G$ be another graph. Notice that the space ${\rm Hom}(S,G)$ is also a retract of the space ${\rm Hom}(T,G)$. Indeed, we have that $i^* r^* = (ri)^* ={\rm id}$.

\begin{lem}
Suppose that a graph $S$ is a retract of a graph $T$. If $S$ is an HT-graph, then so is $T$.
\end{lem}
\begin{proof}
Let $G$ be a graph. Suppose that ${\rm Hom}(T,G)$ is $n$-connected. By Lemma 5, we have that ${\rm Hom}(S,G)$ is $n$-connected. Since $S$ is an HT-graph and $\chi(S) = \chi(T)$, we have
$$\chi(G) > n+ \chi(S) = n + \chi(T).$$
This implies that $T$ is an HT-graph.
\end{proof}

We are now ready to prove Theorem 3. We actually prove the following which clearly implies Theorem 3.

\begin{thm}
Let $T$ be a graph with $\chi(T) = n$. If $T$ has a subgraph isomorphic to $K_n$, then $T$ is an HT-graph.
\end{thm}
\begin{proof}
By the hypothesis, there are graph homomorphisms $i: K_n \rightarrow T$ and $r: T \rightarrow K_n$. Note that any graph homomorphism from $K_n$ to $K_n$ is an isomorphism. Thus we have $(ri)^{-1} ri = {\rm id}_{K_n}$, and hence $K_n$ is a retract of $T$. Since $K_n$ is an HT-graph, it follows from Lemma 6 that $T$ is an HT-graph. 
\end{proof}

\section{Proof of Theorem 4}

\begin{figure}[b]
\begin{center}
The graph homomorphism $f$.

\begin{picture}(90,80)(0,-10)
\put(0,30){\circle*{3}} \put(15,0){\circle*{3}} \put(15,60){\circle*{3}} \put(30,45){\circle*{3}} \put(30,15){\circle*{3}} \put(60,15){\circle*{3}}
\put(60,45){\circle*{3}} \put(75,60){\circle*{3}} \put(75,0){\circle*{3}} \put(90,30){\circle*{3}}
\put(0,30){\line(1,2){15}} \put(0,30){\line(1,-2){15}} \put(15,60){\line(1,-1){15}} \put(15,0){\line(1,1){15}} \put(30,15){\line(0,1){30}} \put(30,15){\line(1,0){30}} \put(30,45){\line(1,0){30}} \put(60,15){\line(0,1){30}} \put(60,15){\line(1,-1){15}} \put(60,45){\line(1,1){15}} \put(75,0){\line(1,2){15}}
\put(75,60){\line(1,-2){15}}

\put(-10,27){2} \put(13,47){1} \put(13,6){1} \put(30,4){3} \put(30,49){2} \put(55,4){2} \put(55,49){1} \put(72,47){3} \put(72,6){3} \put(93,27){1}
\end{picture}

{\bf Figure 2.}
\end{center}
\end{figure}

We start with the following lemma which allows us to prove that $(K,\gamma_1)$ is an SWT-graph. This is known, but we give the proof for reader's convenience.

\begin{prop}[Kozlov, Proposition 6.1.5 of \cite{Koz1}]
Let $A$, $B$, and $C$ be flipping $\mathbb{Z}_2$-graphs satisfying the following:
\begin{itemize}
\item[(a)] $A$ and $C$ are SWT-graphs.
\item[(b)] $\chi(A) = \chi(C)$.
\item[(c)] There are $\mathbb{Z}_2$-equivariant graph homomorphisms $f:A \rightarrow B$ and $g:B \rightarrow C$.
\end{itemize}
Then $B$ is an SWT-graph.
\end{prop}
\begin{proof}
By the condition (c), we have $\chi(A) \leq \chi(B) \leq \chi(C)$. Since $\chi(A) = \chi(C)$, we have $\chi(A) = \chi(B) = \chi(C)$. Let $n$ be an integer with $n \geq 2$. By the condition (c), we have that there are $\mathbb{Z}_2$-maps ${\rm Hom}(C,G) \rightarrow {\rm Hom}(B,G)$ and ${\rm Hom}(B,G) \rightarrow {\rm Hom}(A,G)$. This implies
$$n-\chi(C) = h({\rm Hom}(C,G)) \leq h({\rm Hom}(B,G) \leq h({\rm Hom}(A,G)) = n- \chi(A)$$
since $A$ and $C$ are SWT-graphs. Hence we have $h({\rm Hom}(B,G)) = n- \chi(B)$.
\end{proof}

\begin{cor}
Let $K$ and $\gamma_1$ be those described in Section 1. Then the flipping $\mathbb{Z}_2$-graph $(K,\gamma_1)$ is an SWT-graph.
\end{cor}

\newpage
\begin{center}
\thispagestyle{empty}
\input{Pic1.tex}
\end{center}

\begin{proof}
Here we regard $K$ as the $\mathbb{Z}_2$-graph by the involution $\gamma_1$. Then there are $\mathbb{Z}_2$-equivariant graph homomorphisms $C_5 \rightarrow K$ and $K \rightarrow C_5$. Since $C_5$ with the reflection is an SWT-graph, we have that $(K,\gamma_1)$ is an SWT-graph by Proposition 8.
\end{proof}

Therefore, for the proof of Theorem 4, it suffices to prove the following.

\begin{prop}
The flipping $\mathbb{Z}_2$-graph $(K,\gamma_2)$ is not an SWT-graph.
\end{prop}

\begin{proof}
Here we regard $K$ as a $\mathbb{Z}_2$-graph with the involution $\gamma_2$. Let $f:K \rightarrow K_3$ be the graph homomorphism described in Figure 2. It is enough to show that $f$ and $f \circ \gamma_2$ belong the same connected component of $\Hom (K,K_3)$. In fact, if $f$ and $f \circ \gamma_2$ are contained in the same connected component, there is a $\mathbb{Z}_2$-equivariant map from $S^1$ to ${\rm Hom}(K,K_3)$, and hence $w_1(\Hom(K,K_3)) \neq 0$. This implies that $(K,\gamma_2)$ is not an SWT-graph.

Consider the following condition for graph homomorphisms $\varphi, \psi : K \rightarrow K_3$.

\begin{itemize}
\item[$(*)$] There is a vertex $x$ of $K$ such that $\varphi(v) = \psi (v)$ for any $v \in V(K) \setminus \{ x\}$.
\end{itemize}
If the pair of graph homomorphisms $\varphi$, $\psi$ satisfies the condition $(*)$, then the map $\eta: V(K) \rightarrow 2^{V(K_3)} \setminus \{ \emptyset \}$ defined by $\eta(v) = \{ \varphi (v), \psi (v)\}$ is a multi-homomorphism such that $\varphi \leq \eta$ and $\psi \leq \eta$, and hence $\varphi$ and $\psi$ belong to the same connected component. Since the front and  back of each arrow in Figure 3 satisfy the condition $(*)$, $f$ and $f \circ \gamma_2$ belong to the same connected component of $\Hom (K,K_3)$.
\end{proof}

\begin{rem}
Here we write SWT$'$-graphs to mean Stiefel-Whitney test graphs in the sense of \cite{DS1} (see Section 2). Since $h(\Hom (K,K_3)) \geq 1$, the $\mathbb{Z}_2$-graph $(K,\gamma_2)$ is not an SWT$'$-graph. On the other hand $(K,\gamma_1)$ is an SWT$'$-graph since SWT-graphs are SWT$'$-graphs. Hence even if we replace the definition of Stiefel-Whitney test graphs of Theorem 4, the assertion is also true.
\end{rem}

As was mentioned in Section 2, if the flipping $\mathbb{Z}_2$-graph $(T,\gamma)$ is an SWT-graph, then $T$ is an HT-graph. We assert that the converse does not hold.

\begin{cor}
There is a flipping $\mathbb{Z}_2$-graph $(T,\gamma)$ such that $T$ is an HT-graph but $(T,\gamma)$ is not an SWT-graph.
\end{cor}
\begin{proof}
Let $K,\gamma_1,\gamma_2$ be those of Theorem 4. Then since $(K,\gamma_1)$ is an SWT-graph, $K$ is an HT-graph. But $(K,\gamma_2)$ is not an SWT-graph.
\end{proof}

\end{document}

%% file: Pic1.tex
\setlength\unitlength{0.5truecm}
\begin{picture}(15,44)(0,-3)
\put(6.6,-1){\bf Figure 3.}

\put(2.8,42.8){$f$}

\put(12.2,42.8){$f \circ \gamma_2$}

\put(1,38.5){\circle*{0.2}} \put(0,40.5){\circle*{0.2}}
\put(1,42.5){\circle*{0.2}}
\put(2,39.5){\circle*{0.2}}
\put(2,41.5){\circle*{0.2}}
\put(4,41.5){\circle*{0.2}}
\put(4,39.5){\circle*{0.2}}
\put(5,38.5){\circle*{0.2}}
\put(6,40.5){\circle*{0.2}}
\put(5,42.5){\circle*{0.2}}

\put(0,40.5){\line(1,2){1}}
\put(0,40.5){\line(1,-2){1}}
\put(1,38.5){\line(1,1){1}}
\put(1,42.5){\line(1,-1){1}}
\put(2,39.5){\line(1,0){2}}
\put(2,39.5){\line(0,1){2}}
\put(2,41.5){\line(1,0){2}}
\put(4,39.5){\line(0,1){2}}
\put(4,39.5){\line(1,-1){1}}
\put(5,38.5){\line(1,2){1}}
\put(4,41.5){\line(1,1){1}}
\put(5,42.5){\line(1,-2){1}}

\put(1,33){\circle*{0.2}}
\put(0,35){\circle*{0.2}}
\put(1,37){\circle*{0.2}}
\put(2,34){\circle*{0.2}}
\put(2,36){\circle*{0.2}}
\put(4,36){\circle*{0.2}}
\put(4,34){\circle*{0.2}}
\put(5,33){\circle*{0.2}}
\put(6,35){\circle*{0.2}}
\put(5,37){\circle*{0.2}}

\put(0,35){\line(1,2){1}}
\put(0,35){\line(1,-2){1}}
\put(1,33){\line(1,1){1}}
\put(1,37){\line(1,-1){1}}
\put(2,34){\line(1,0){2}}
\put(2,34){\line(0,1){2}}
\put(2,36){\line(1,0){2}}
\put(4,34){\line(0,1){2}}
\put(4,34){\line(1,-1){1}}
\put(5,33){\line(1,2){1}}
\put(4,36){\line(1,1){1}}
\put(5,37){\line(1,-2){1}}

\put(1,27.5){\circle*{0.2}}
\put(0,29.5){\circle*{0.2}}
\put(1,31.5){\circle*{0.2}}
\put(2,28.5){\circle*{0.2}}
\put(2,30.5){\circle*{0.2}}
\put(4,30.5){\circle*{0.2}}
\put(4,28.5){\circle*{0.2}}
\put(5,27.5){\circle*{0.2}}
\put(6,29.5){\circle*{0.2}}
\put(5,31.5){\circle*{0.2}}

\put(0,29.5){\line(1,2){1}}
\put(0,29.5){\line(1,-2){1}}
\put(1,27.5){\line(1,1){1}}
\put(1,31.5){\line(1,-1){1}}
\put(2,28.5){\line(1,0){2}}
\put(2,28.5){\line(0,1){2}}
\put(2,30.5){\line(1,0){2}}
\put(4,28.5){\line(0,1){2}}
\put(4,28.5){\line(1,-1){1}}
\put(5,27.5){\line(1,2){1}}
\put(4,30.5){\line(1,1){1}}
\put(5,31.5){\line(1,-2){1}}

\put(1,22){\circle*{0.2}}
\put(0,24){\circle*{0.2}}
\put(1,26){\circle*{0.2}}
\put(2,23){\circle*{0.2}}
\put(2,25){\circle*{0.2}}
\put(4,25){\circle*{0.2}}
\put(4,23){\circle*{0.2}}
\put(5,22){\circle*{0.2}}
\put(6,24){\circle*{0.2}}
\put(5,26){\circle*{0.2}}

\put(0,24){\line(1,2){1}}
\put(0,24){\line(1,-2){1}}
\put(1,22){\line(1,1){1}}
\put(1,26){\line(1,-1){1}}
\put(2,23){\line(1,0){2}}
\put(2,23){\line(0,1){2}}
\put(2,25){\line(1,0){2}}
\put(4,23){\line(0,1){2}}
\put(4,23){\line(1,-1){1}}
\put(5,22){\line(1,2){1}}
\put(4,25){\line(1,1){1}}
\put(5,26){\line(1,-2){1}}

\put(1,16.5){\circle*{0.2}}
\put(0,18.5){\circle*{0.2}}
\put(1,20.5){\circle*{0.2}}
\put(2,17.5){\circle*{0.2}}
\put(2,19.5){\circle*{0.2}}
\put(4,19.5){\circle*{0.2}}
\put(4,17.5){\circle*{0.2}}
\put(5,16.5){\circle*{0.2}}
\put(6,18.5){\circle*{0.2}}
\put(5,20.5){\circle*{0.2}}

\put(0,18.5){\line(1,2){1}}
\put(0,18.5){\line(1,-2){1}}
\put(1,16.5){\line(1,1){1}}
\put(1,20.5){\line(1,-1){1}}
\put(2,17.5){\line(1,0){2}}
\put(2,17.5){\line(0,1){2}}
\put(2,19.5){\line(1,0){2}}
\put(4,17.5){\line(0,1){2}}
\put(4,17.5){\line(1,-1){1}}
\put(5,16.5){\line(1,2){1}}
\put(4,19.5){\line(1,1){1}}
\put(5,20.5){\line(1,-2){1}}

\put(1,11){\circle*{0.2}}
\put(0,13){\circle*{0.2}}
\put(1,15){\circle*{0.2}}
\put(2,12){\circle*{0.2}}
\put(2,14){\circle*{0.2}}
\put(4,14){\circle*{0.2}}
\put(4,12){\circle*{0.2}}
\put(5,11){\circle*{0.2}}
\put(6,13){\circle*{0.2}}
\put(5,15){\circle*{0.2}}

\put(0,13){\line(1,2){1}}
\put(0,13){\line(1,-2){1}}
\put(1,11){\line(1,1){1}}
\put(1,15){\line(1,-1){1}}
\put(2,12){\line(1,0){2}}
\put(2,12){\line(0,1){2}}
\put(2,14){\line(1,0){2}}
\put(4,12){\line(0,1){2}}
\put(4,12){\line(1,-1){1}}
\put(5,11){\line(1,2){1}}
\put(4,14){\line(1,1){1}}
\put(5,15){\line(1,-2){1}}

\put(1,5.5){\circle*{0.2}}
\put(0,7.5){\circle*{0.2}}
\put(1,9.5){\circle*{0.2}}
\put(2,6.5){\circle*{0.2}}
\put(2,8.5){\circle*{0.2}}
\put(4,8.5){\circle*{0.2}}
\put(4,6.5){\circle*{0.2}}
\put(5,5.5){\circle*{0.2}}
\put(6,7.5){\circle*{0.2}}
\put(5,9.5){\circle*{0.2}}

\put(0,7.5){\line(1,2){1}}
\put(0,7.5){\line(1,-2){1}}
\put(1,5.5){\line(1,1){1}}
\put(1,9.5){\line(1,-1){1}}
\put(2,6.5){\line(1,0){2}}
\put(2,6.5){\line(0,1){2}}
\put(2,8.5){\line(1,0){2}}
\put(4,6.5){\line(0,1){2}}
\put(4,6.5){\line(1,-1){1}}
\put(5,5.5){\line(1,2){1}}
\put(4,8.5){\line(1,1){1}}
\put(5,9.5){\line(1,-2){1}}

\put(1,0){\circle*{0.2}}
\put(0,2){\circle*{0.2}}
\put(1,4){\circle*{0.2}}
\put(2,1){\circle*{0.2}}
\put(2,3){\circle*{0.2}}
\put(4,3){\circle*{0.2}}
\put(4,1){\circle*{0.2}}
\put(5,0){\circle*{0.2}}
\put(6,2){\circle*{0.2}}
\put(5,4){\circle*{0.2}}

\put(0,2){\line(1,2){1}}
\put(0,2){\line(1,-2){1}}
\put(1,0){\line(1,1){1}}
\put(1,4){\line(1,-1){1}}
\put(2,1){\line(1,0){2}}
\put(2,1){\line(0,1){2}}
\put(2,3){\line(1,0){2}}
\put(4,1){\line(0,1){2}}
\put(4,1){\line(1,-1){1}}
\put(5,0){\line(1,2){1}}
\put(4,3){\line(1,1){1}}
\put(5,4){\line(1,-2){1}}

\put(11,38.5){\circle*{0.2}}
\put(10,40.5){\circle*{0.2}}
\put(11,42.5){\circle*{0.2}}
\put(12,39.5){\circle*{0.2}}
\put(12,41.5){\circle*{0.2}}
\put(14,41.5){\circle*{0.2}}
\put(14,39.5){\circle*{0.2}}
\put(15,38.5){\circle*{0.2}}
\put(16,40.5){\circle*{0.2}}
\put(15,42.5){\circle*{0.2}}

\put(10,40.5){\line(1,2){1}}
\put(10,40.5){\line(1,-2){1}}
\put(11,38.5){\line(1,1){1}}
\put(11,42.5){\line(1,-1){1}}
\put(12,39.5){\line(1,0){2}}
\put(12,39.5){\line(0,1){2}}
\put(12,41.5){\line(1,0){2}}
\put(14,39.5){\line(0,1){2}}
\put(14,39.5){\line(1,-1){1}}
\put(15,38.5){\line(1,2){1}}
\put(14,41.5){\line(1,1){1}}
\put(15,42.5){\line(1,-2){1}}

\put(11,33){\circle*{0.2}}
\put(10,35){\circle*{0.2}}
\put(11,37){\circle*{0.2}}
\put(12,34){\circle*{0.2}}
\put(12,36){\circle*{0.2}}
\put(14,36){\circle*{0.2}}
\put(14,34){\circle*{0.2}}
\put(15,33){\circle*{0.2}}
\put(16,35){\circle*{0.2}}
\put(15,37){\circle*{0.2}}

\put(10,35){\line(1,2){1}}
\put(10,35){\line(1,-2){1}}
\put(11,33){\line(1,1){1}}
\put(11,37){\line(1,-1){1}}
\put(12,34){\line(1,0){2}}
\put(12,34){\line(0,1){2}}
\put(12,36){\line(1,0){2}}
\put(14,34){\line(0,1){2}}
\put(14,34){\line(1,-1){1}}
\put(15,33){\line(1,2){1}}
\put(14,36){\line(1,1){1}}
\put(15,37){\line(1,-2){1}}

\put(11,27.5){\circle*{0.2}}
\put(10,29.5){\circle*{0.2}}
\put(11,31.5){\circle*{0.2}}
\put(12,28.5){\circle*{0.2}}
\put(12,30.5){\circle*{0.2}}
\put(14,30.5){\circle*{0.2}}
\put(14,28.5){\circle*{0.2}}
\put(15,27.5){\circle*{0.2}}
\put(16,29.5){\circle*{0.2}}
\put(15,31.5){\circle*{0.2}}

\put(10,29.5){\line(1,2){1}}
\put(10,29.5){\line(1,-2){1}}
\put(11,27.5){\line(1,1){1}}
\put(11,31.5){\line(1,-1){1}}
\put(12,28.5){\line(1,0){2}}
\put(12,28.5){\line(0,1){2}}
\put(12,30.5){\line(1,0){2}}
\put(14,28.5){\line(0,1){2}}
\put(14,28.5){\line(1,-1){1}}
\put(15,27.5){\line(1,2){1}}
\put(14,30.5){\line(1,1){1}}
\put(15,31.5){\line(1,-2){1}}

\put(11,22){\circle*{0.2}}
\put(10,24){\circle*{0.2}}
\put(11,26){\circle*{0.2}}
\put(12,23){\circle*{0.2}}
\put(12,25){\circle*{0.2}}
\put(14,25){\circle*{0.2}}
\put(14,23){\circle*{0.2}}
\put(15,22){\circle*{0.2}}
\put(16,24){\circle*{0.2}}
\put(15,26){\circle*{0.2}}

\put(10,24){\line(1,2){1}}
\put(10,24){\line(1,-2){1}}
\put(11,22){\line(1,1){1}}
\put(11,26){\line(1,-1){1}}
\put(12,23){\line(1,0){2}}
\put(12,23){\line(0,1){2}}
\put(12,25){\line(1,0){2}}
\put(14,23){\line(0,1){2}}
\put(14,23){\line(1,-1){1}}
\put(15,22){\line(1,2){1}}
\put(14,25){\line(1,1){1}}
\put(15,26){\line(1,-2){1}}

\put(11,16.5){\circle*{0.2}}
\put(10,18.5){\circle*{0.2}}
\put(11,20.5){\circle*{0.2}}
\put(12,17.5){\circle*{0.2}}
\put(12,19.5){\circle*{0.2}}
\put(14,19.5){\circle*{0.2}}
\put(14,17.5){\circle*{0.2}}
\put(15,16.5){\circle*{0.2}}
\put(16,18.5){\circle*{0.2}}
\put(15,20.5){\circle*{0.2}}

\put(10,18.5){\line(1,2){1}}
\put(10,18.5){\line(1,-2){1}}
\put(11,16.5){\line(1,1){1}}
\put(11,20.5){\line(1,-1){1}}
\put(12,17.5){\line(1,0){2}}
\put(12,17.5){\line(0,1){2}}
\put(12,19.5){\line(1,0){2}}
\put(14,17.5){\line(0,1){2}}
\put(14,17.5){\line(1,-1){1}}
\put(15,16.5){\line(1,2){1}}
\put(14,19.5){\line(1,1){1}}
\put(15,20.5){\line(1,-2){1}}

\put(11,11){\circle*{0.2}}
\put(10,13){\circle*{0.2}}
\put(11,15){\circle*{0.2}}
\put(12,12){\circle*{0.2}}
\put(12,14){\circle*{0.2}}
\put(14,14){\circle*{0.2}}
\put(14,12){\circle*{0.2}}
\put(15,11){\circle*{0.2}}
\put(16,13){\circle*{0.2}}
\put(15,15){\circle*{0.2}}

\put(10,13){\line(1,2){1}}
\put(10,13){\line(1,-2){1}}
\put(11,11){\line(1,1){1}}
\put(11,15){\line(1,-1){1}}
\put(12,12){\line(1,0){2}}
\put(12,12){\line(0,1){2}}
\put(12,14){\line(1,0){2}}
\put(14,12){\line(0,1){2}}
\put(14,12){\line(1,-1){1}}
\put(15,11){\line(1,2){1}}
\put(14,14){\line(1,1){1}}
\put(15,15){\line(1,-2){1}}

\put(11,5.5){\circle*{0.2}}
\put(10,7.5){\circle*{0.2}}
\put(11,9.5){\circle*{0.2}}
\put(12,6.5){\circle*{0.2}}
\put(12,8.5){\circle*{0.2}}
\put(14,8.5){\circle*{0.2}}
\put(14,6.5){\circle*{0.2}}
\put(15,5.5){\circle*{0.2}}
\put(16,7.5){\circle*{0.2}}
\put(15,9.5){\circle*{0.2}}

\put(10,7.5){\line(1,2){1}}
\put(10,7.5){\line(1,-2){1}}
\put(11,5.5){\line(1,1){1}}
\put(11,9.5){\line(1,-1){1}}
\put(12,6.5){\line(1,0){2}}
\put(12,6.5){\line(0,1){2}}
\put(12,8.5){\line(1,0){2}}
\put(14,6.5){\line(0,1){2}}
\put(14,6.5){\line(1,-1){1}}
\put(15,5.5){\line(1,2){1}}
\put(14,8.5){\line(1,1){1}}
\put(15,9.5){\line(1,-2){1}}

\put(11,0){\circle*{0.2}}
\put(10,2){\circle*{0.2}}
\put(11,4){\circle*{0.2}}
\put(12,1){\circle*{0.2}}
\put(12,3){\circle*{0.2}}
\put(14,3){\circle*{0.2}}
\put(14,1){\circle*{0.2}}
\put(15,0){\circle*{0.2}}
\put(16,2){\circle*{0.2}}
\put(15,4){\circle*{0.2}}

\put(10,2){\line(1,2){1}}
\put(10,2){\line(1,-2){1}}
\put(11,0){\line(1,1){1}}
\put(11,4){\line(1,-1){1}}
\put(12,1){\line(1,0){2}}
\put(12,1){\line(0,1){2}}
\put(12,3){\line(1,0){2}}
\put(14,1){\line(0,1){2}}
\put(14,1){\line(1,-1){1}}
\put(15,0){\line(1,2){1}}
\put(14,3){\line(1,1){1}}
\put(15,4){\line(1,-2){1}}

\put(7.5,2){\vector(1,0){1}}

\put(3,5){\vector(0,-1){1}}
\put(3,10.5){\vector(0,-1){1}}
\put(3,16){\vector(0,-1){1}}
\put(3,21.5){\vector(0,-1){1}}
\put(3,27){\vector(0,-1){1}}
\put(3,32.5){\vector(0,-1){1}}
\put(3,38){\vector(0,-1){1}}

\put(13,4){\vector(0,1){1}}
\put(13,9.5){\vector(0,1){1}}
\put(13,15){\vector(0,1){1}}
\put(13,20.5){\vector(0,1){1}}
\put(13,26){\vector(0,1){1}}
\put(13,31.5){\vector(0,1){1}}
\put(13,37){\vector(0,1){1}}

\put(0.3,1.8){$3$}
\put(0.9,3.2){$1$}
\put(0.9,0.3){$2$}
\put(2,3.2){$3$}
\put(2,0.3){$1$}
\put(3.7,3.2){$1$}
\put(3.7,0.3){${\bf 3}$}
\put(4.8,3.2){$3$}
\put(4.8,0.3){$1$}
\put(5.4,1.8){$2$}

\put(0.3,7.3){$3$}
\put(0.9,8.7){$1$}
\put(0.9,5.8){$2$}
\put(2,8.7){$3$}
\put(2,5.8){$1$}
\put(3.7,8.7){$1$}
\put(3.7,5.8){$2$}
\put(4.8,8.7){$3$}
\put(4.8,5.8){${\bf 1}$}
\put(5.4,7.3){$2$}

\put(0.3,12.8){$3$}
\put(0.9,14.2){$1$}
\put(0.9,11.3){$2$}
\put(2,14.2){$3$}
\put(2,11.3){$1$}
\put(3.7,14.2){$1$}
\put(3.7,11.3){$2$}
\put(4.8,14.2){$3$}
\put(4.8,11.3){$3$}
\put(5.4,12.8){${\bf 2}$}

\put(0.3,18.3){$3$}
\put(0.9,19.7){$1$}
\put(0.9,16.8){$2$}
\put(2,19.7){${\bf 3}$}
\put(2,16.8){$1$}
\put(3.7,19.7){$1$}
\put(3.7,16.8){$2$}
\put(4.8,19.7){$3$}
\put(4.8,16.8){$3$}
\put(5.4,18.3){$1$}

\put(0.3,23.8){$3$}
\put(0.9,25.2){$1$}
\put(0.9,22.3){$2$}
\put(2,25.2){$2$}
\put(2,22.3){${\bf 1}$}
\put(3.7,25.2){$1$}
\put(3.7,22.3){$2$}
\put(4.8,25.2){$3$}
\put(4.8,22.3){$3$}
\put(5.4,23.8){$1$}

\put(0.3,29.3){$3$}
\put(0.9,30.7){$1$}
\put(0.9,27.8){${\bf 2}$}
\put(2,30.7){$2$}
\put(2,27.8){$3$}
\put(3.7,30.7){$1$}
\put(3.7,27.8){$2$}
\put(4.8,30.7){$3$}
\put(4.8,27.8){$3$}
\put(5.4,29.3){$1$}

\put(0.3,34.8){${\bf 3}$}
\put(0.9,36.2){$1$}
\put(0.9,33.3){$1$}
\put(2,36.2){$2$}
\put(2,33.3){$3$}
\put(3.7,36.2){$1$}
\put(3.7,33.3){$2$}
\put(4.8,36.2){$3$}
\put(4.8,33.3){$3$}
\put(5.4,34.8){$1$}

\put(0.3,40.3){$2$}
\put(0.9,41.7){$1$}
\put(0.9,38.8){$1$}
\put(2,41.7){$2$}
\put(2,38.8){$3$}
\put(3.7,41.7){$1$}
\put(3.7,38.8){$2$}
\put(4.8,41.7){$3$}
\put(4.8,38.8){$3$}
\put(5.4,40.3){$1$}

\put(10.3,1.8){$3$}
\put(10.9,3.2){$1$}
\put(10.9,0.3){$2$}
\put(12,3.2){$3$}
\put(12,0.3){$1$}
\put(13.7,3.2){${\bf 2}$}
\put(13.7,0.3){$3$}
\put(14.8,3.2){$3$}
\put(14.8,0.3){$1$}
\put(15.4,1.8){$2$}

\put(10.3,7.3){$3$}
\put(10.9,8.7){$1$}
\put(10.9,5.8){$2$}
\put(12,8.7){$3$}
\put(12,5.8){$1$}
\put(13.7,8.7){$2$}
\put(13.7,5.8){$3$}
\put(14.8,8.7){${\bf 1}$}
\put(14.8,5.8){$1$}
\put(15.4,7.3){$2$}

\put(10.3,12.8){$3$}
\put(10.9,14.2){${\bf 2}$}
\put(10.9,11.3){$2$}
\put(12,14.2){$3$}
\put(12,11.3){$1$}
\put(13.7,14.2){$2$}
\put(13.7,11.3){$3$}
\put(14.8,14.2){$1$}
\put(14.8,11.3){$1$}
\put(15.4,12.8){$2$}

\put(10.3,18.3){${\bf 1}$}
\put(10.9,19.7){$2$}
\put(10.9,16.8){$2$}
\put(12,19.7){$3$}
\put(12,16.8){$1$}
\put(13.7,19.7){$2$}
\put(13.7,16.8){$3$}
\put(14.8,19.7){$1$}
\put(14.8,16.8){$1$}
\put(15.4,18.3){$2$}

\put(10.3,23.8){$1$}
\put(10.9,25.2){$2$}
\put(10.9,22.3){${\bf 3}$}
\put(12,25.2){$3$}
\put(12,22.3){$1$}
\put(13.7,25.2){$2$}
\put(13.7,22.3){$3$}
\put(14.8,25.2){$1$}
\put(14.8,22.3){$1$}
\put(15.4,23.8){$2$}

\put(10.3,29.3){$1$}
\put(10.9,30.7){$2$}
\put(10.9,27.8){$3$}
\put(12,30.7){$3$}
\put(12,27.8){${\bf 2}$}
\put(13.7,30.7){$2$}
\put(13.7,27.8){$3$}
\put(14.8,30.7){$1$}
\put(14.8,27.8){$1$}
\put(15.4,29.3){$2$}

\put(10.3,34.8){$1$}
\put(10.9,36.2){$2$}
\put(10.9,33.3){$3$}
\put(12,36.2){${\bf 1}$}
\put(12,33.3){$2$}
\put(13.7,36.2){$2$}
\put(13.7,33.3){$3$}
\put(14.8,36.2){$1$}
\put(14.8,33.3){$1$}
\put(15.4,34.8){$2$}

\put(10.3,40.3){$1$}
\put(10.9,41.7){${\bf 3}$}
\put(10.9,38.8){$3$}
\put(12,41.7){$1$}
\put(12,38.8){$2$}
\put(13.7,41.7){$2$}
\put(13.7,38.8){$3$}
\put(14.8,41.7){$1$}
\put(14.8,38.8){$1$}
\put(15.4,40.3){$2$}
\end{picture}